\documentclass[a4paper,12pt]{article}
\usepackage{amsmath,amsfonts,amssymb,amscd}
\usepackage{indentfirst,graphicx,epsfig}
\usepackage{graphicx,psfrag}
\usepackage{amsthm, amsfonts, color}
\usepackage[bookmarksnumbered, plainpages]{hyperref}
\usepackage[numbers,sort&compress]{natbib}
\input{epsf}

\title{Proper connection numbers of \\
complementary graphs\thanks{Supported by NSFC No.11371205 and
PCSIRT.} }
\author{\small{Fei Huang, Xueliang Li, Shujing Wang}\\
{\small  Center for Combinatorics and LPMC-TJKLC}\\
{\small Nankai University, Tianjin 300071, China}\\\makeatletter
{\small Email: huangfei06@126.com; lxl@nankai.edu.cn;  \newcommand\figcaption{\def\@captype{figure}\caption}
wang06021@126.com} }  \newcommand\tabcaption{\def\@captype{table}\caption}
\date{}\makeatother
\newtheorem{theorem}{Theorem}[section]
\newtheorem{defi}{Definition}[section]
\newtheorem{lemma}[theorem]{Lemma}
\newtheorem{pro}[theorem]{Proposition}
\newtheorem{coro}[theorem]{Corollary}

\begin{document}

\maketitle

\begin{abstract}
  A path $P$ in an edge-colored graph $G$ is called a proper path if no two adjacent edges of $P$ are colored the same, and $G$ is proper connected if every two vertices of $G$ are connected by a proper path in $G$. The proper connection number of a connected graph $G$, denoted by $pc(G)$, is the minimum number of colors that are needed to make $G$ proper connected. In this paper, we investigate the proper connection number of the complement of graph $G$ according to some constraints of $G$ itself. Also, we characterize the graphs on $n$ vertices that have proper connection number $n-2$. Using this result, we give a Nordhaus-Gaddum-type theorem for the proper connection number. We prove that if $G$ and $\overline{G}$ are both connected, then $4\le pc(G)+pc(\overline{G})\le n$, and the only graph attaining the upper bound is the tree with maximum degree $\Delta=n-2$.
\\[2mm]


\end{abstract}

\section{Introduction}

In this paper we are concerned with simple connected finite graphs. We
follow the terminology and the notation of Bondy and Murty \cite{Bondy}.
The distance between two vertices $u$ and $v$ in a connected graph $G$, denoted by $dist(u,v)$, is the length of a shortest path between them in $G$. The eccentricity of a vertex $v$ in $G$ is defined as $ecc_G(v)=\max\{x\in V(G): dist(v,x)\}$, and the diameter of $G$ denoted by $diam(G)$ is defined as $diam(G)=\max\{x\in V(G):ecc_G(v)\}$.

An edge coloring of a graph $G$ is an assignment $c$ of colors to the edges of $G$, one color to each edge of $G$. If adjacent edges of $G$ are assigned different colors by $c$, then $c$ is a \emph{proper (edge) coloring}. The minimum number of colors needed in a proper coloring of $G$ is referred to as the \emph{chromatic index} of $G$ and denoted by $\chi'(G)$. A path in an edge-colored graph with no two edges sharing the same color is called a \emph{rainbow path}. An edge-colored graph $G$ is said to be \emph{rainbow connected} if every pair of distinct vertices of $G$ is connected by at least one rainbow path in $G$. Such a coloring is called a \emph{rainbow coloring} of the graph. The minimum number of colors in a rainbow coloring of $G$ is referred to as the \emph{rainbow connection number} of $G$ and denoted by $rc(G)$. The concept of rainbow coloring was first introduced by Chartrand et al. in \cite{Chartrand}. In recent years, the rainbow coloring has been extensively studied and has gotten a variety of nice results, see \cite{Chandran,Chartrand2,Kri,Li,Li3} for examples. For more details we refer to a survey paper \cite{Li1} and a book \cite{Li2}.

Inspired by rainbow colorings and proper colorings in graphs, Andrews et al. \cite{Andrews} introduce the concept of proper-path colorings. Let $G$ be an edge-colored graph, where adjacent edges may be colored the same. A path $P$ in $G$ is called a \emph{proper path} if no two adjacent edges of $P$ are colored the same. An edge-coloring $c$ is a \emph{proper-path coloring} of a connected graph $G$ if every pair of distinct vertices $u,v$ of $G$ is connected by a proper $u-v$ path in $G$. A graph with a proper-path coloring is said to be \emph{proper connected}. If $k$ colors are used, then $c$ is referred to as a \emph{proper-path $k$-coloring}. The minimum number of colors needed to produce a proper-path coloring of $G$ is called the \emph{proper connection number} of $G$, denoted by $pc(G)$.

Let $G$ be a nontrivial connected graph of order $n$ and size $m$.
Then the proper connection number of $G$ has the following apparent bounds:
$$1 \le pc(G) \le \min\{\chi'(G), rc(G)\}\le m.$$ Furthermore, $pc(G)=1$ if and only if $G = K_n$ and $pc(G)=m$ if and only if $G=K_{1,m}$ is a star of size $m$.

Among many interesting problems of determining the proper connection numbers of graphs, it is worth to study the proper connection number of $G$ according to some constraints of the complementary graph. In \cite{Li4}, the authors considered this kind of question for rainbow connection number $rc(G)$.

A Nordhaus-Gaddum-type result is a (tight) lower or upper bound on the sum
or product of the values of a parameter for a graph and its complement. The name ``Nordhaus-Gaddum-type" is given because Nordhaus and Gaddum \cite{Nord} first established the type of inequalities for the chromatic number of graphs in 1956. They proved that if $G$ and $\overline{G}$ are complementary graphs on $n$ vertices whose chromatic numbers are $\chi(G)$ and $\chi(\overline{G})$, respectively, then $2\sqrt{n}\le \chi(G)+ \chi(\overline{G})\le n+1$. Since then, many analogous inequalities of other graph parameters have been considered, such as diameter \cite{HFR}, domination number \cite{HFT}, rainbow connection number \cite{CLLL,CLL}, generalized edge-connectivity \cite{LM}, and so on.

The rest of this paper is organised as follows: In section $2$, we list some important known results on proper connection number. In section $3$, we investigate the proper connection number of the complement of graph $G$ according to some constraints of $G$ itself. In section $4$, we first characterize the graphs on $n$ vertices that have proper connection number $n-2$. Using this result, we give a Nordhaus-Gaddum-type theorem for the proper connection number. We prove that if $G$ and $\overline{G}$ are both connected, then $4\le pc(G)+pc(\overline{G})\le n$, and the only graph that attaining the upper bound is the tree with maximum degree $\Delta=n-2$.
\section{Preliminaries}

At the beginning of this section, we list some fundamental results on proper-path coloring which can be found in \cite{Andrews}.

\begin{lemma}\label{lem2.1} If $G$ is a nontrivial connected graph and $H$ is a connected spanning subgraph of $G$, then $pc(G)\le pc(H)$. In particular, $pc(G)\le pc(T)$ for every spanning tree $T$ of $G$.
\end{lemma}

\begin{lemma}\label{lem2.2}
 Let $G$ be a nontrivial connected graph that contains
bridges. If $b$ is the maximum number of bridges incident with a single
vertex in $G$, then $pc(G)\ge b$.
\end{lemma}

\begin{lemma}\label{lem2.3}
If $T$ is a nontrivial tree, then $pc(T)=\chi'(T)=\Delta(T)$.
\end{lemma}

Given a colored path $P =v_1v_2\ldots v_{s-1}v_s$ between any two vertices $v_1$ and $v_s$, we denote by $start(P)$ the color of the first edge in the path, i.e. $c(v_1v_2)$, and by $end(P)$ the last color, i.e. $c(v_{s-1}v_s)$. If $P$ is just the edge $v_1v_s$ then $start(P)=end(P)=c(v_1v_s)$.

\begin{defi}

Let $c$ be an edge-coloring of $G$ that makes $G$ proper connected. We say $G$ has the strong property if for any pair of vertices $u, v\in V(G)$, there exist two proper paths $P_1$, $P_2$ between them (not necessarily disjoint) such that $start(P_1)\neq start(P_2)$  and $end(P_1)\neq end(P_2)$.
\end{defi}

In \cite{Borozan}, the authors studied proper-connection numbers in bipartite graphs. Also, they presented a result which improve the upper bound $\Delta(G)+1$ of $pc(G)$ to the best possible whenever the graph $G$ is bipartite and $2$-connected.

\begin{lemma}\cite{Borozan}\label{lem2.4}
Let $G$ be a graph. If $G$ is bipartite and $2$-connected then $pc(G)=2$ and there exists a $2$-edge-coloring of $G$ such that $G$ has the strong property.
\end{lemma}

Every complete $k$-partite graph $G=K_{n_1,n_2,\ldots,n_k}$ contains
a spanning bipartite subgraph $H=K_{n_1+n_2+\ldots n_{k-1},n_k}$. We
know that $H$ is $2$-connected if $n_k\ge 2$ and $k\ge 3$.
Therefore, we have the following result.

\begin{coro}\label{cor2.5}
Every complete $k$-partite graph $(k\ge 3)$ except for the complete
graph $K_k$ has proper connection number two, and there exists a
$2$-edge-coloring $c$ of $G$ such that $G$ has the strong property.
\end{coro}

For general $2$-connected graphs, Borozan et al. \cite{Borozan} gave
a tight upper bound for the proper connection number.
\begin{lemma}\label{lem2.6}
Let $G$ be a graph. If $G$ is $2$-connected then $pc(G)\leq 3$ and there exists a $3$-edge-coloring $c$ of $G$ such that $G$ has the strong property.
\end{lemma}

\begin{coro}\label{cor2.7}
Let $H=G\cup\{v_1\}\cup\{v_2\}$. If there is a proper-path
$k$-coloring $c$ of $G$ such that $G$ has the strong property, then
$pc(H)\leq k$ as long as $v_1,v_2$ are not isolated vertices of $H$.
Moreover, we have that $pc(G\cup\{v_1\})\le k$ under the same
assumption.
\end{coro}
\begin{proof}
Let $u_1\in N_H(v_1)$ and $u_2\in N_H(v_2)$, and let $1, 2$ be two
used colors. If $u_1=u_2$, we extend the coloring $c$ of $G$ to the
whole graph $H$ by assigning color $1$ to $u_1v_1$, and $2$ to
$u_2v_2$. Otherwise, $u_1\neq u_2$. Since $G$ is proper connected,
there exists a proper path $P$ of $G$ connecting $u_1$ and $u_2$. We
assign an used color that is distinct from $start(P)$ to $u_1v_1$,
and an used color distinct from $end(P)$ to $u_2v_2$. In both cases,
$v_1$ and $v_2$ are connected by a proper path. For any $w\in V(G)$,
we can also easily check that $w$ and $v_i$ ($i=1,2$) are connected
by a proper-path since $G$ has the strong property. Hence $pc(H)\le
k$. The conclusion $pc(G\cup\{v_1\})\le k$ follows directly using
the analysis above. Hence we complete the proof.
\end{proof}

\section{Proper connection number of complementary graph}

We first investigate the proper connection numbers of connected complement graphs of graphs with diameter at least $4$.
\begin{theorem}\label{th2.1}
If $G$ is a connected graph with $diam(G)\geq 4$, then $pc(\overline{G})=2$.
\end{theorem}
\begin{figure}
  \centering
  \includegraphics{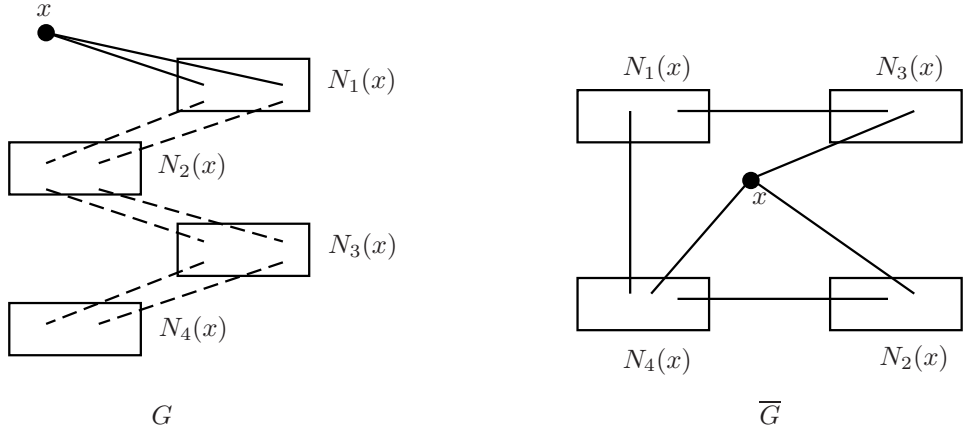}\\
  \caption{$G$ and $\overline{G}$ with  $diam(G)\geq 4$}
\end{figure}
\begin{proof}
First of all, we see that $\overline{G}$ must be connected, since otherwise, $diam(G)\leq 2$, contradicting the condition $diam(G)\geq 4$.
We choose a vertex $x$ with $ecc_G(x)=diam(G)$. Let $N_i(x)=\{v: dist(x,v)=i\}$ where $0\leq i\leq 3$ and $N_4 (x)=\{v: dist(x,v)\geq 4\}$. So $N_0=\{x\}$ and $N_1=N_G(x)$. In the rest of our paper, we use $N_i$ instead of $N_i(x)$ for convenient. By the definition of $N_i$, we know that in $\overline{G}$, there is a spanning subgraph $G^*$ such that $G^*[N_1\cup N_3](G^*[N_1\cup N_4], G^*[N_2\cup N_4])$ is a complete bipartite graph(see Fig. 1). We give $G^*$ an edge-coloring as follows: we first give the color 1 to the edges $xu$ for $u\in N_3$ and to all edges between $N_1$ and $N_4$; next we give the color 2 to all the remaining edges. Now we prove that this coloring is a proper-path coloring.

It is obvious that for any $u\in N_i$ and $v\in N_j$ with $i\not=j$, $u, v$ are connected by a proper path. So it suffices to show that for any $u,v\in N_i$, there is a proper path connecting them in $G^*$. For $i=1$, let $P=ux_3xx_4v$ where $x_3\in N_3$ and $x_4\in N_4$. Clearly, $P$ is a proper path. Similarly, there is a proper path connecting any two vertices $u,v\in N_3$ or $N_4$. For $i=2$, let $Q=uxx_3x_1x_4v$, where $x_1\in N_1, x_3\in N_3$ and $x_4\in N_4$. One can see that $Q$ is a proper path. Hence we have that $G^*$ is proper connected, i.e., $pc(G^*)\leq 2$. Together with the fact that $\overline{G}$ is not complete, we have that $1\not=pc(\overline{G})\leq pc(G^*)\leq 2$. Hence we have $pc(\overline{G})=2$.
\end{proof}

\begin{theorem} For a connected noncomplete graph $G$, if $\overline{G}$ does not belong to the following two cases: (i) $diam(\overline{G})=2,3$, (ii) $\overline{G}$ contains exactly two connected components and one of them is trivial, then $pc(G)=2$.
\end{theorem}
\begin{proof}
If $\overline{G}$ is connected, we know that $diam(\overline{G})\ge 4$. Hence  $pc(G)=2$ clearly holds by Theorem \ref{th2.1}. If $\overline{G}$ is disconnected. Suppose that $\overline{G}_i$ $(1\le i\le h)$ are the connected components of $\overline{G}$ with $n_i=|V(\overline{G}_i)|$. Then $G$ contains a spanning subgraph $K_{n_1,n_2,\ldots,n_h}$. By the assumption, $\overline{G}$ has either at least three connected components or exactly two nontrivial components. Then we have $pc(G)=2$ from Lemma \ref{lem2.4} and Corollary \ref{cor2.5}.
\end{proof}

If $diam(G)=3$, we have the following theorem for the proper connection number of $\overline{G}$.

\begin{theorem}
Let $G$ be a connected graph with $diam(G)=3$ and $x$ the vertex of $G$ such that $ecc_G(x)=3$ (see Fig. 2). We have $pc(\overline{G})=2$ for the two cases (i) $n_1=n_2=n_3=1$, (ii) $n_2=1, n_3\geq 2$. For the remaining cases, $pc(\overline{G})$ may be very large. Furthermore, if $G$ is triangle-free, then $pc(\overline{G})=2$.
\end{theorem}
\begin{figure}
  \centering
 \scalebox{0.6}[0.5]{\includegraphics{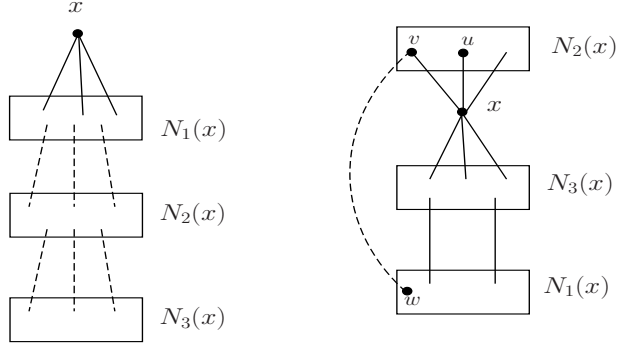}}\\
  \caption{$G$ and $\overline{G}$ with  $diam(G)= 3$}
\end{figure}

\begin{proof}
If $n_1=n_2=n_3=1$, then $G$ is a 4-path $P_4$, and so
$pc(\overline{G})=pc(P_4)=2$. Thus we consider the case that $n_2=1,
n_3\geq 2$. One can see that $\overline{G}$ contains a spanning
subgraph $G^*$ such that $G^*[N_0\cup N_1\cup N_3]$ is a complete
bipartite graph $K_{1+n_1,n_3}$. By Lemma \ref{lem2.4}, we know that
$pc(G^*)=2$. By corollary \ref{cor2.7}, we have that
$pc(\overline{G})=2$ since $n_2=1$. The remaining cases are: (1)
$n_1>1,n_2=n_3=1$, and (2) $n_2\geq 2$.

If $n_2=n_3=1$ and $n_1>1$, let $N_2=\{x_2\}$ and $N_3=\{x_3\}$, One
can see that $xv, x_2v\not \in E(\overline{G})$ for any $v\in N_1$.
Let $n_1'=|\{v\in N_1: N_G(v)\cap N_1=\emptyset\}|$. One can see
that there are $n_1'$ cut edges in $\overline{G}$ that is adjacent
to $x_3$. Hence, by Lemma \ref{lem2.2}. we have that
$pc(\overline{G})\geq n_1'$.

Furthermore, if $G$ is triangle-free, then $N_1$ is an independent set in $G$, and so a clique in $\overline{G}$. Now give color 1 to $x_2x$ and $x_3v$ for any $v\in N_1$ and color 2 to $xx_3$ and $uv$ for any $u,v\in N_2$. One can see that this coloring is a proper-path 2-coloring, thus $pc(\overline{G})=2$.

If $n_2\geq 2$, let $n_2'=|\{v\in N_2: deg_{\overline{G}}(v)=1\}|$,
and then there are $n_2'$ cut edges in $\overline{G}$ that is
adjacent to $x$. Hence, by Lemma \ref{lem2.2}, we have that
$pc(\overline{G})\geq n_2'$.

Furthermore, if $G$ is triangle-free, then $N_1$ is an independent set in $G$, and so a clique in $\overline{G}$. We give $\overline{G}$ an edge-coloring as follows: we give color 1 to $xx_2$ for any $x_2\in N_2$ and $x_1x_3$ for any $x_1\in N_1, x_3\in N_3$ and give color 2 to all the other edges in $\overline{G}$. Now we prove that this coloring is a proper-path 2-coloring.

It is obvious that for any $u\in N_i, v\in N_j$ with $i\not=j$, there is a proper path connecting them. It suffices to show that for any $u, v\in N_2$ or $N_3$ with $uv\not\in E(\overline{G})$, there is a proper path between them. In fact, as $G$ is triangle-free, if $uv\in E(G)$, one can see that there is a vertex $w\in N_1$ such that $wu\in E(G)$ and $wv\not\in E(G)$. Thus $P=uxx_3wv$ is a proper path connecting $u$ and $v$ in $\overline{G}$ where $x_3\in N_3$. Similarly, we can see that for any $u,v\in N_3$, there is a proper path between them. Thus we have that this coloring is a proper-path 2-coloring. So $pc(\overline{G})=2$.
\end{proof}

The following two corollaries clearly hold.
\begin{coro}\label{cor3.4}
For a graph $G$, if $\overline{G}$ is triangle-free and $diam(\overline{G})=3$, then $pc(G)=2$.
\end{coro}

If $G$ is acyclic, it is apparent that $G$ is triangle-free. From Theorem \ref{th2.1} and corollary \ref{cor3.4}, we have the following result.
\begin{coro}
If $G$ is a tree with $diam(G)\ge 3$, then $pc(\overline{G})=2$.
\end{coro}

\begin{theorem}\label{th3.6}
Let $G$ be a triangle-free graph with $diam(G)=2$. If $\overline{G}$ is connected, then $pc(\overline{G})=2$.
\end{theorem}

\begin{proof}
We choose a vertex $x$ with $ecc_G(x)=2$, and $N_i=\{v: dist(x,v)=i\}$ for $i=0,1,2$. One can see that $N_0=\{x\}, N_1=N_G(x), N_2=V\setminus (N_1\cup N_0)$. As $G$ is triangle-free, it is obvious that $N_1$ is a clique in $\overline{G}$. Note that $\overline{G}$ is connect, thus $|N_1|>1$ and there is at least one edge $uv\in E(\overline{G})$ such that $u\in N_1$ and $v\in N_2$.

We give $\overline{G}$ an edge-coloring as follows: we give color $1$ to the edges between $N_1$ and $N_2$, and give color $2$ to all the other edges in $\overline{G}$. Now we prove that this coloring is a proper-path coloring. First, we can easily find that there are proper paths between $x$ and any other vertices.  Also, there are  proper paths between $v$ and vertices in $N_1$.  For any $y\in N_2\setminus\{v\}$ and $z\in N_1$, if $N_{\overline{G}}(y)\cap N_1\neq\emptyset$, let $w\in N_{\overline{G}}(y)\cap N_1$.  Then $ywz$ is a proper path between $y$ and $z$. Otherwise, $N_{\overline{G}}(y)\cap N_1=\emptyset$. We claim that $y$ is adjacent to all the other vertices in $N_2$. In fact, for any vertex $w\in N_2\setminus y$, there exists a vertex $w'\in N_1$ such that $ww'\in E(G)$.
Since $yw'\in  E(G)$, we know that $yw\in E(\overline{G})$.  Especially, we know that $yv\in E(\overline{G})$. Then  $yvuz$ is a proper path between $y$ and $z$. Next consider $x_2, x_2'\in N_2$ such that $x_2x_2'\notin E(\overline{G})$. Since $x_2,x_2'\in N_2$, there are $x_1,x_1'\in N_1$ such that $x_1x_2,x_1'x_2'\in E(G)$. As $G$ is triangle-free, one can see that $x_1x_2', x_2x_1'\in E(\overline{G})$. So we have that $x_2x_1'x_1x_2'$ is a proper path connecting $x_1$ and $x_1'$.  Hence we have that $pc(\overline{G})=2$.
\end{proof}

\begin{pro}\label{pr3.1}
If $G$ is triangle-free and contains two connected components one of which is trivial, then $pc(\overline{G})=2$.
\end{pro}
\begin{proof}
Let $G_1$ and $G_2$ be the two components of $G$ such that $V(G_1)=\{v\}$.
Then $\overline{G}=\overline{G_1}\vee \overline{G_2}$, where ``$\vee$" is the join of two graphs, that is, vertex $v$ is adjacent to all the other vertices in $\overline{G}$. If $\overline{G_2}$ is connected, then $pc(\overline{G_2})=2$ from Theorem \ref{th2.1}, Corollary \ref{cor3.4} and Theorem \ref{th3.6}. Hence, we can get that $pc(\overline{G})=2$. Otherwise, $\overline{G_2}$ is disconnected.  Since $G$ is triangle-free, we know that  $\overline{G_2}$ has two connected components, and both of them are cliques of $\overline{G_2}$. We can easily find a proper-path 2-coloring for $\overline{G}$. Hence $pc(\overline{G})=2$. We complete the proof.
\end{proof}

\begin{theorem} For a connected noncomplete graph $G$, if $\overline{G}$ is triangle-free, then $pc(G)=2$.
\end{theorem}

\begin{proof}
We consider the following two cases:

\textbf{Case 1}. $\overline{G}$ is connected.

The result holds for the case $diam(\overline{G})\le 4$ from Theorem
\ref{th2.1}, the case $diam(\overline{G})=3$ from Corollary
\ref{cor3.4} and the case $diam(\overline{G}) = 2$ from Theorem
\ref{th3.6}.

\textbf{Case 2}. $\overline{G}$ is disconnected.

The result holds for the case that $\overline{G}$ contains two
connected components with one of them trivial from Proposition
\ref{pr3.1}, and holds for the remaining case from Lemma
\ref{lem2.4} and Corollary \ref{cor2.5}.
\end{proof}

\section{Nordhaus-Gaddum-type theorem for proper connection number of graphs}

Firstly, we characterize the graphs on $n$ vertices that have proper
connection number $n-2$. This result is crucial to investigate the
Nordhaus-Gaddum-type result for the proper connection number of $G$.
We use $C_n,S_n$ to denote the cycle and the star graph on $n$
vertices, respectively, and use $T(a,b)$ to denote the double star
in which the degrees of its two center vertices are $a$ and $b$
respectively. For a nontrivial graph $G$ for which $G+uv=G+xy$ for
every two pairs $\{u, v\}$, $\{x, y\}$ of nonadjacent vertices of
$G$, the graph $G+e$ is obtained from $G$ by adding the edge $e$
joining two nonadjacent vertices of $G$.

\begin{theorem}\label{NG1}
Let $G$ be a connected graph on $n$ vertices. Then $pc(G)=n-2$ if
and only if $G$ is one of the following $6$ graphs:
$T(2,n-2),C_3,C_4,C_4+e,S_4+e,S_5+e$.
\end{theorem}

\begin{proof}
Let $G$ be one of the above 6 graphs. We can easily check that
$pc(G)=n-2$. So it remains to verify the converse. If $G$ is
acyclic, from Lemma \ref{lem2.3}, we know that $G=T(2,n-2)$. Suppose
that $G$ contains cycles. Let $G^*$ be a spanning unicycle subgraph
of $G$ such that the cycle $C$ in $G^*$ is the longest cycle in $G$.
Without loss of generality, assume that $C=v_1v_2\ldots v_kv_1$ and
$d_G(v_1)\geq d_G(v_i)$ for $i=2,3\cdots, l$. As $pc(C)=2$ for all
$k\geq 4$, we can see that $pc(G)\leq pc(G^*)\leq 2+n-k<n-2$ if
$k>4$, contradicting with the fact that $pc(G)=n-2$. So we only need
to consider that $k=3$ or $k=4$.

If $k=4$, let $G_1=G^*-v_1v_2$. One can see that $G_1$ is a spanning
tree of $G$, with $\Delta(G_1)\leq n-3$ unless $n=4$. So by Lemma
\ref{lem2.1}, $pc(G)\leq pc(G_1)\leq n-3$, contradicting the fact
that $pc(G)=n-2$. So $n=4$ and $G^*=C_4$. Hence, $G=C_4$ or
$G=C_4+e$ since the longest cycle of $G$ is of length $4$.

Now we consider the case $k=3$. Let $c$ be an edge coloring of $G^*$
such that the cut edges are colored by $n-3$ distinct colors. If
$n\ge 6$, that is, $G^*$ has more than three cut edges, choose three
colors that have been used on the cut edges, say $1,2,3$. Let
$c(v_1v_2)=1$, $c(v_2v_3)=2$ and $c(v_3v_1)=3$. We know that $G^*$
is proper connected under edge-coloring $c$. Hence, $pc(G)\leq
pc(G^*)\le n-3$, contradicting the fact that $pc(G)=n-2$. So $n\leq
5$. If $n=5$, one can see that $G\cong S_5+e$, since otherwise,
there is a spanning $P_5$ in $G$, then $pc(G)\leq pc(P_5)=2$, a
contradiction. If $n=4$, one can see that $G\cong S_4+e$, since
otherwise there exists a cycle of length 4 in $G$ which
contradicting the assumption $k=3$. If $n=3$, $G\cong C_3$ as
$pc(G)=1$ if and only if $G$ is complete graph. Hence we have that
if $k=3$,$G=C_3$, or $G=S_4+e$, or $G=S_5+e$.
\end{proof}

We know that if $G$ is a connected graph with $n$ vertices, then the
number of the edges in $G$ must be at least $n-1$. If both $G$ and
$\overline{G}$ are connected, then $n$ is at least $4$, and
$\Delta(G)\le n-2$. Therefore, we know that $2\le pc(G)\le n-2$.
Similarly, $2\le pc(\overline{G})\le n-2$. Hence we can obtain that
$4\le pc(G)+pc(\overline{G})\le 2(n-2)$. For $n=4$, we can easily
get that $pc(G)+pc(\overline{G})=4$ if $G$ and $\overline{G}$  are
connected. In the rest of the paper, we always assume that all
graphs have at least $5$ vertices, and both $G$ and $\overline{G}$
are connected.
\begin{lemma}\label{lem4.2}
If $n=5$, then
\begin{equation*}
pc(G)+pc(\overline{G})=
\left\{
  \begin{array}{ll}
   5 & \hbox{if $G\cong T(2,n-2)$ or $\overline{G}\cong T(2,n-2)$,} \\
   4 & \hbox{otherwise}.
\end{array}
\right.
\end{equation*}
\end{lemma}

\begin{proof}
If $G\cong T(2,n-2)$ or $\overline{G}\cong T(2,n-2)$, then
$pc(G)+pc(\overline{G})=5$ clearly holds. From Theorem \ref{NG1}, we
know that $T(2,n-2)$ is the only graph on $5$ vertices that has
proper connection number $3$.  Since $2\le pc(G)\le 3$ and $2\le
pc(\overline{G})\le 3$, then all the other graphs considered here on
$5$ vertices has  proper connection number $2$. Hence
$pc(G)+pc(\overline{G})=4$ if $G\ncong T(2,n-2)$ and
$\overline{G}\ncong T(2,n-2)$.
\end{proof}

\begin{theorem}
$pc(G)+pc(\overline{G})\leq n $ for $n\geq 5$, and the equality
holds if and only if $G\cong T(2,n-2)$ or $\overline{G}\cong
T(2,n-2)$.
\end{theorem}
\begin{proof}
By Lemma \ref{lem4.2}, we can see that the result holds if $n=5$. So
we consider $n\geq 6$. If $G\cong T(2,n-2)$, $\overline{G}$ contains
a spanning subgraph $H$ that is obtained by attaching a pendent edge
to the complete bipartite graph $K_{2,n-3}$. Hence we have that
$pc(G)=n-2$ and $pc(\overline{G})=2$. The result clearly holds.
Similarly, we can also get $pc(G)+pc(\overline{G})=n$ if
$\overline{G}\cong T(2,n-2)$. To prove our conclusion, we only need
to show that $pc(G)+pc(\overline{G})< n $ if $G\ncong T(2,n-2)$ and
$\overline{G}\ncong T(2,n-2)$. Under this assumption, we know that
$2\le pc(G)\le n-3$ and $2\le pc(\overline{G})\le n-3$ by Theorem
\ref{NG1}.

Suppose first that both $G$ and $\overline{G}$ are $2$-connected.
For $n=6$, we claim that $pc(G)=2$. Suppose that the circumference
of $G$ is $k$. If $k=6$, one has that $pc(G)\leq pc(C_6)=2$. If the
$k=4$, one can see that $G$ contains a spanning $K_{2,4}$,
contradicting the fact that $\overline{G}$ is $2$-connected. Assume
that $G$ contains a $5$-cycle $C=v_1v_2v_3v_4v_5v_1$, we know that
the vertex $v_6$ is adjacent to two vertices that is nonadjacent in
$C$, say $v_1,v_3$. We give a $2$-edge coloring $c$ to this spanning
subgraph $H$ of $G$ as follows. Let
$c(v_1v_2)=c(v_3v_4)=c(v_5v_1)=c(v_3v_6)=1$, and the other edges
color $2$.  One can see that $G$ is proper connected. Hence
$pc(G)=2$. Similarly, $pc(\overline{G})=2$. So we have that
$pc(G)+pc(\overline{G})\leq 2+2<6$. For $n\geq 7$, by Lemma
\ref{lem2.6}, we know that $pc(G)\le 3$ and $pc(\overline{G})\le 3$,
and so $pc(G)+pc(\overline{G})\le 6$, and therefore
$pc(G)+pc(\overline{G})<n$ clearly holds.

 Now We consider the case that at least one of $G$ and $\overline{G}$ has cut vertices. Without loss of generality, suppose that $G$ has cut vertices. We distinguish the following three cases.

\textbf{Case 1}. $G$ has a cut vertex $u$ such that $G-u$ has at least three components.

Let $G_1, G_2,\ldots,G_k$ ($k\ge 3$) be the components of $G-u$ and
let $n_i$ be the number of vertices of $G_i$ for $1\le i\le k$ with
$n_1\le n_2\le\ldots \le n_k$. From the definition of
$\overline{G}$, we know that $\overline{G}-u$ contains a spanning
complete $k$-partite graph $K_{n_1,n_2,\ldots, n_k}$. Since
$\Delta(G)\le n-2$, then $n_k\ge 2$. From Corollary \ref{cor2.5},
$pc(\overline{G}-u)=2$, and there exists a $2$-edge-coloring $c$ of
$\overline{G}-u$ that makes it proper connected with the strong
property. Hence $pc(\overline{G})\le 2$ by Corollary \ref{cor2.7}.
Together with the fact that  $pc(G)\le n-3$, we can get the result
$pc(G)+pc(\overline{G})<n$.

\textbf{Case 2}. Each cut vertex $u$ of $G$ satisfies that $G-u$ has only two components.

Let $G_1, G_2$ be the two components of $G-u$, and let $n_i$ be the
number of vertices of $G_i$ for $i=1,2$ with $n_1\le n_2$.

\textbf{Subcase 2.1}. $n_1 \ge 2$.

Then $\overline{G}-u$ contains a spanning $2$-connected bipartite
graph $K_{n_1,n_2}$. From Lemma \ref{lem2.4}, we know that
$pc(\overline{G}-u)=2$ and there exists a $2$-edge-coloring $c$ of
$\overline{G}-u$ that makes it proper connected with the strong
property. So by Corollary \ref{cor2.7}, $pc(\overline{G})\leq 2$. We
can get the result that $pc(G)+pc(\overline{G})<n$.

\textbf{Subcase 2.2}. $n_1=1$, that is, each cut vertex is incident
with a pendent edge.

Let $u_1v_1,u_2v_2,\ldots,u_lv_l$ be the pendent edges of $G$ such
that $v_i$ is the pendent vertices for $1\le i\le l$. The pendent
edges are pairwise disjoint. Let $H$ be the graph obtained from $G$
by deleting all the pendent vertices. Then $H$ must be
$2$-connected. By Lemma \ref{lem2.6}, we know that $pc(H)\le 3$ and
there exists a $3$-edge-coloring $c$ of $\overline{G}-u$ that makes
it proper connected with the strong property.

If $l\ge 2$, we know that $\overline{G}-\{u_1,u_2\}$ contains a
spanning bipartite subgraph  $K_{2,n-4}$ with two parts
$X=\{v_1,v_2\}$ and  $Y=V(G)\setminus \{u_1,v_1,u_2,v_2\}$. Since
$v_1u_2, v_2u_1\notin E(G)$, we know that $v_1u_2, v_2u_1\in
E(\overline{G})$. Then by Lemma \ref{lem2.4} and Corollary
\ref{cor2.7}, we have that $pc(\overline{G})\leq 2$. By using the
fact that $pc(G)\le n-3$, we have that $pc(G)+pc(\overline{G})<n$.

If $l=1$, by Lemma \ref{lem2.6} and Corollary \ref{cor2.7}, one has
that $pc(G)\leq pc(H)\leq 3$. Therefore we have
$pc(G)+pc(\overline{G})\leq n$. Now we prove that the equality
cannot be attained. Note that $d_{\overline{G}}(v_1)=n-2$. We know
that $\overline{G}$ contains $T(2,n-2)$ as a proper spanning
subgraph. Set $N_{\overline{G}}(v_1)= \{x_1,\cdots,
x_{n-2}\}=V(G)\setminus \{u_1, v_1\}$. Without loss of generality,
assume that $x_1u_1\not\in E(G)$. So $x_1u_1\in E(\overline{G})$. If
there is a vertex $x_j$ ($2\le j\le n-2$) that is adjacent to $x_1$
in $\overline{G}$, assume without loss of generality that $x_1x_2\in
E(\overline{G})$. Let $c(v_1x_1)=1,c(x_1x_2)=2,
c(v_1x_2)=c(x_1u_1)=3$ and $c(v_1x_i)=i-2$ for $i=3,4\cdots, n-2$.
One can see that $\overline{G}$ is proper connected. If there is a
vertex $x_j$ ($2\le j\le n-2$) that is adjacent to $u_1$ in
$\overline{G}$, assume without loss of generality that $x_2u_2\in
E(\overline{G})$. Let $c(v_1x_i)=i-2$ for $i=3,4\cdots, n-2$ and
$c(x_1x_1)=c(u_1x_2)=1,c(v_1x_2)=c(x_1u_1)=2$. One can also see that
$\overline{G}$ is proper connected. If there are two vertex $x_j,
x_k$ ($2\le j<k\le n-2$) such that $x_jx_k\in E(\overline{G})$,
without loss of generality, assume that $x_2x_3\in E(\overline{G})$.
Let $c(v_1x_i)=i-2$ for $i=4,\cdots, n-2$, $c(v_1x_1)=c(v_1x_2)=1,
c(v_1x_3)=c(x_1u_1)=2$ and $c(x_2x_3)=3$. We can check that
$\overline{G}$ is proper connected. Hence we have that
$pc(\overline{G})\leq \max\{3,n-4\}$. For $n\geq 7$, we can get that
$pc(G)+pc(\overline{G})\leq 3+n-4=n-1<n$. For $n=6$, as $H$ is a
2-connected graph with 5 vertices, one can see that $H$ contains a
spanning $C_5$ or a spanning $K_{2,3}$. Hence we can easy get that
$pc(G)=pc(H)=2$. So we have $pc(G)+pc(\overline{G})\leq 2+3=5<6$.
Our proof is complete.
\end{proof}


\begin{thebibliography}{10}
\bibitem{Andrews} E. Andrews, E. Laforge, C. Lumduanhom, P. Zhang, On proper-path colorings in graphs, J. Combin. Math. Combin. Comput, to appear.

\bibitem{Bondy} J.A. Bondy, U.S.R. Murty, Graph Theory, Springer,
New York, 2008.

\bibitem{Borozan}V. Borozan, S. Fujita, A. Gerek,  C. Magnant, Y. Manoussakis, L. Monteroa, Z. Tuza, Proper connection of graphs, Discrete Math. 312 (17) (2012), 2550-2560.

\bibitem{Chandran} L. Chandran, A. Das, D.Rajendraprasad, N. Varma, Rainbow
connection number and connected dominating sets, J. Graph Theory 71 (2012), 206-218.

\bibitem{Chartrand} G. Chartrand, G. L. Johns, K. A. McKeon, P. Zhang, Rainbow connection in graphs, Math Bohemica 133 (1) (2008), 5-98.

\bibitem{Chartrand2} G. Chartrand, G.L. Johns, K.A. McKeon, P. Zhang, The rainbow connectivity of a graph, Networks 54 (2) (2009),75-81.

\bibitem{CLLL} L. Chen, X. Li, H. Lian, Nordhaus-Gaddum-type theorem for rainbow connection number of graphs, Graphs \& Combin. 29 (5) (2013), 1235-1247.
\bibitem{CLL} L. Chen, X. Li, M. Liu, Nordhaus-Gaddum-type bounds for rainbow vertex-connection number of a graph, Utilitas Math. 86 (2011), 335-340.

\bibitem{HFR} F. Harary, R.W. Robinson, The diameter of a graph and its complement, Amer. Math. Monthly 92 (1985), 211-212.

\bibitem{HFT} F. Harary, T.W. Haynes, Nordhaus-Gaddum inequalities for domination in graphs, Discrete Math. 155 (1996), 99-105.

\bibitem{Kri} M. Kriveleich, R. Yuster, The rainbow connection of a graph is (at most) reciprocal to its minimum degree, J. Graph Theory 71 (2012), 206-218.

\bibitem{Li} H. Li, X. Li, S. Liu, Rainbow connection of graphs with diameter $2$, Discrete Math. 312 (2012), 1453-1457.

\bibitem{LM}   X. Li, Y. Mao, Nordhaus-Gaddum-type results for the generalized edge-connectivity of graphs, Discrete Appl. Math. 185 (2015), 102-112.

\bibitem{Li3}  X. Li, Y. Shi, Rainbow connection in $3$-connected graphs, Graphs \& Combin. 29 (5) (2013), 1471-1475.

\bibitem{Li1}  X. Li, Y. Shi, Y. Sun, Rainbow connections of graphs: A survey, Graphs \& Combin. 29 (2013), 1-38.

\bibitem{Li2} X. Li, Y. Sun, Rainbow Connections of Graphs, Springer Briefs in Math., Springer, New York, 2012.

\bibitem{Li4}  X. Li, Y. Sun, Rainbow connection numbers of complementary graphs, Util. Math. 86 (2011), 23-31.

\bibitem{Nord} E.A. Nordhaus, J.W. Gaddum, On complementary graphs, Amer. Math. Monthly 63 (1956), 175-177.

\end{thebibliography}
\end{document}